\documentclass[10pt,a4paper,notitlepage, oneside]{scrartcl}

\usepackage[utf8]{inputenc}
\usepackage{subfigure}
\usepackage{amsthm,amsmath,amssymb}
\usepackage{mathtools}
\usepackage[ruled,linesnumbered, vlined]{algorithm2e}
\usepackage{authblk}
\usepackage{tikz}
\usepackage{todonotes}
\usepackage[short,nocomma]{optidef}

\newtheorem{theorem}{Theorem}
\newtheorem*{reptheorem}{Theorem \ref{thm:sds_np_perfect_graphs}}
\newtheorem{lemma}[theorem]{Lemma}
\newtheorem{corollary}[theorem]{Corollary}

\newtheorem{assumption}[theorem]{Assumption}
\newtheorem*{assumption*}{Assumption}
\theoremstyle{definition}
\newtheorem{definition}[theorem]{Definition}
\newtheorem{idea}[theorem]{Explanation}

\newcommand{\best}[1]{\operatorname{best}\{#1\}}
\newcommand{\instate}{{2}}
\newcommand{\domstate}{{1}}
\newcommand{\notdomstate}{{0}}
\newcommand{\NP}{\textsf{NP}\xspace}
\newcommand{\Pp}{\textsf{P}\xspace}

\newcommand{\minVCprob}{\textsc{minVertexCover}\xspace}
\newcommand{\VCprob}{\textsc{VertexCover}\xspace}

\newcommand{\sdset}{simultaneous dominating set\xspace}
\newcommand{\sdsets}{simultaneous dominating sets\xspace}
\newcommand{\Sdsets}{Simultaneous dominating sets\xspace}
\newcommand{\SDset}{SD-set\xspace}
\newcommand{\SDsets}{SD-sets\xspace}
\newcommand{\simdom}{simultaneously dominate\xspace}
\newcommand{\simdoms}{simultaneously dominates\xspace}
\newcommand{\simdomed}{simultaneously dominated\xspace}
\newcommand{\simdoming}{simultaneously dominating\xspace}
\newcommand{\sdsprob}{\textsc{SimultaneousDominatingSet}\xspace}
\newcommand{\minsdsprob}{\textsc{minSimultaneousDominatingSet}\xspace}

\newcommand{\rSDS}{respecting SD-set\xspace}

\newcommand{\rsds}{respecting simultaneous dominating set\xspace}
\newcommand{\rsdss}{respecting simultaneous dominating sets\xspace}

\newcommand{\cf}{\textit{cf.}}
\newcommand{\ie}{\textit{i.e.}}
\newcommand{\CV}{\operatorname{CV}}
\newcommand{\NCV}{\operatorname{NCV}}

\newcommand{\colorone}{cyan}
\newcommand{\coloronetext}{blue\xspace}
\newcommand{\colortwo}{yellow}
\newcommand{\colortwotext}{yellow\xspace}

\usepackage{tabularx,environ}
\makeatletter
\newcommand{\problemtitle}[1]{\gdef\@problemtitle{#1}}
\newcommand{\probleminput}[1]{\gdef\@probleminput{#1}}
\newcommand{\problemquestion}[1]{\gdef\@problemquestion{#1}}
\newcommand{\problemcomplexity}[1]{\gdef\@problemcomplexity{#1}}
\NewEnviron{decisionproblem}{
	\problemtitle{}\probleminput{}\problemquestion{}\problemcomplexity{}
	\BODY
	\par\addvspace{.5\baselineskip}
	\noindent\@problemtitle
	\begin{addmargin}[.8cm]{0cm}
		\hspace*{-.4cm}\textbf{Instance:} \@probleminput\\%
		\hspace*{-.4cm}\textbf{Question:} \@problemquestion%
		\ifthenelse{\equal{\@problemcomplexity}{}}{}{
			\\
			\hspace*{-.4cm}\textbf{Complexity:} \@problemcomplexity
		}
	\end{addmargin}
	
	\par\addvspace{.5\baselineskip}
}
\makeatother

\usepackage{enumitem}
\renewcommand{\labelenumi}{(\roman{enumi})}
\newenvironment{thmenum}
{\begin{enumerate}[label=\upshape(\roman*),ref=\thetheorem.(\roman*),noitemsep]}{\end{enumerate}}

\newenvironment{thmenumnonum}
{\begin{enumerate}[label=\upshape(\roman*),ref=(\roman*),noitemsep]}{\end{enumerate}}

\usepackage[numbers]{natbib}

\author[1]{Sebastian S. Johann}
\author[1]{Sven O. Krumke}
\author[1]{Manuel Streicher}
\affil[1]{Technische Universit\"at Kaiserslautern}
\title{Simultaneously Dominating all Spanning Trees of a Graph}

\begin{document}

\maketitle

\begin{abstract}%
	We investigate the problem of simultaneously dominating all spanning trees of a given graph.
	We prove that on $2$-connected graphs, a subset of the vertices dominates all spanning trees of the graph if and only if it is a vertex cover. 
	Using this fact we present an exact algorithm that finds a simultaneous dominating set of minimum size using an oracle for finding a minimum vertex cover.
	The algorithm can be implemented to run in polynomial time on several graph classes, such as bipartite or chordal graphs.
	We prove that there is no polynomial time algorithm that finds a minimum simultaneous dominating set on perfect graphs unless $\Pp=\NP$.
	Finally, we provide a $2$-approximation algorithm for finding a minimum simultaneous dominating set.
	
\end{abstract}

\section{Introduction}
A dominating set in a graph $G$ is a subset $S\subseteq V(G)$ of the vertices such that every vertex that is not contained in $S$ has a neighbor in $S$.
The dominating set problem, that aims to find a minimum dominating set in a graph, was formalized by Berge in 1958 \cite{berge1958} and Ore in 1962 \cite{ore1962theory}.
Since then several variants of the dominating set problem have been studied.
One example is the independent dominating set problem, in which additionally no two vertices in the searched set may be adjacent, \cf~\cite{GODDARD2013839}.
Another example is the total dominating set problem in which every vertex has to be adjacent to a vertex in the searched set, \cf~\cite{HENNING200932}.
The dominating set problem is well known to be \NP-complete, see \cite{Garey79}, and so are most of it variants.

In this paper we consider a variation of the dominating set problem in which we seek to simultaneously dominate all spanning trees of a graph.
The concept of simultaneous domination in graphs was independently introduced by Sampathkumar in \cite{sampathkumar:89} under the name global domination and by Brigham and Dutton in \cite{BRIGHAM1990127} who used the term factor domination.
Following~\cite{BRIGHAM1990127}, given a graph $G$ and a partition of its edge set~$E_1,\dots, E_k$, a subset of the vertices is a \emph{factor dominating set} if it is dominating for all graphs $(V(G),E_i)$.
Whereas a susbet of the vertices is a \emph{global dominating set} if it is a subset of the vertices which is dominating in $G$ and its complement.
Later on the term factor domination has also been used for subsets of the vertices that dominate some set of arbitrary subgraphs of $G$ on the same vertex set, see \textit{e.g.}~\cite{dankelmann:06} and~\cite{caro:14}. 
In our studies we use the term \emph{simultaneous domination}, as in our definition the edge sets of the subgraphs are not required to be disjoint.
In the simultaneous dominating set problem regarded here we are given a graph $G$ and we aim for a minimum subset of vertices that is a dominating set in every spanning tree of $G$. As we only regard simultaneous domination of all spanning trees, in the following we often omit \emph{all spanning trees} in order to shorten notation.

Simultaneous domination of all spanning trees has previously not been regarded in the literature. 
We prove that in a $2$-connected graph $G$ a set $S\subseteq V(G)$ dominates all spanning trees if and only if it is a vertex cover. 
On general graphs we prove that the size of a minimum size vertex cover and a minimum size \sdset may differ by a factor of two and give an example that this bound is tight.
We utilize the relation of \sdsets to vertex covers in order to derive an algorithm that finds a minimum size simultaneous dominating set. 
The algorithm works on the block graph of a graph and uses an oracle for \VCprob.
It can be implemented to run in polynomial time on bipartite graphs, chordal graphs and graphs with bounded treewidth. 
The polynomial running times strongly rely on the fact that \VCprob is polynomial time solvable on these classes. 
However, the arguments used are not applicable to all graph classes on which \VCprob is solvable in polynomial time. 
In particular, it is well known that \VCprob is polynomial time solvable on perfect graphs, \cf~\cite{schrijver-book}. 
Yet, one of our main results proves that the same does not hold for simultaneous domination.
\begin{reptheorem}
	\sdsprob is \NP-complete when restricted to perfect graphs.	
\end{reptheorem}
The theorem proves that \VCprob and \sdsprob are not equivalent from the point of view of complexity theory. 
Another direct consequence of the theorem is, that although \sdsprob is polynomial time solvable on $2$-connected, perfect graphs, it is \NP-hard on all perfect graphs. 
In a sense, one could say that the problem significantly simplifies when restricting it to $2$-connected graphs. 
This is a property that is rarely seen among graph theoretic problems, as polynomial time solvability for problems on $2$-connected graphs often implies polynomial time solvability of the corresponding problem on all graphs.

It is well known that \minVCprob may be approximated by a factor of $2$, \cf~\cite{schrijver-book}, as well as that it cannot be approximated by any constant factor smaller than $2$, provided the unique games conjecture holds, \cf~\cite{khot:08}. 
Here, we provide a $2$-approximation for \minsdsprob that is based on LP-rounding. Note that the $2$-approximability is not immediately implied from the $2$-approximability of \minVCprob, as the size of a minimum simultaneous dominating set and the size of a minimum vertex cover may differ by up to a factor of $2$, \cf~Theorem~\ref{sds_spts:all_spanning_trees:thm:2timesVC}.

\paragraph{Outline}
After we state some basic definitions in Section~\ref{sec::preliminaries}, we focus on the characterization and the complexity of \sdsprob in Section~\ref{complexity}. 
In Section~\ref{sec::sds_alg} we present an algorithm to find a minimum size \sdset on general graphs using an oracle for computing a minimum size vertex cover.
Afterwards, we show in Section~\ref{sec::polynomial} that we can solve \sdsprob in polynomial time on bipartite graphs, on chordal graphs and on graphs of bounded treewidth.
Finally, we present a $2$-approximation algorithm for \minsdsprob in Section \ref{sec::sds_approx}.

\section{Preliminaries}\label{sec::preliminaries}
Most of our notation is standard graph terminology which can be found in \cite{Diestel17}.
For an introduction to graph theory from the algorithmic point of view we refer to \cite{Krumke2012}.
Nevertheless, we recall some basic notations in the following.
All graphs under consideration are undirected and simple.
For a graph $G$, we refer to its vertex set by $V(G)$ and to its edge set by $E(G)$.
For an edge joining vertices $u,v\in V(G)$ we write $uv$.
For a subset $S\subseteq V(G)$ we denote by $G[S]$ the graph \emph{induced} by $S$, that has vertex set $S$ and contains all edges in $G$ joining vertices in $S$.
Further, we write $G-S$ for the graph $G[V(G)\setminus S]$ and for $E\subseteq E(G)$ we write $G-E$ for for the graph with vertex set $V(G)$ and edge set $E(G)\setminus E$.
To simplify notation we write $G-v$ and $G-e$ instead of $G-\{v\}$ and $G-\{e\}$ for $v\in V(G)$ and $e\in E(G)$.

A \emph{path} $P=u_0\ldots, u_k$ is a graph with vertex set $V(P)=\{u_0,u_1,\ldots,u_k\}$ and edge set $E(P)=\{u_0u_1,u_1u_2,\ldots,u_{k-1}u_k\}$, where all the $u_i$ are distinct. 
A graph $G$ is called \emph{connected} if any two vertices are linked by a path. As it facilitates arguments in this contribution we make the following assumption for connected graphs.
\begin{assumption}\label{ass:conn_2}
	Any connected graph contains at least two vertices.
\end{assumption}
For a connected graph $G$ a vertex $v\in V(G)$ is called a \emph{cutvertex} if $G-v$ is not connected.
We call a graph \emph{$2$-connected} if it has at least three vertices and does not contain a cutvertex.
A \emph{block} of~$G$ is a maximal connected subgraph of $G$ that does not contain a cutvertex.
The \emph{block graph} of $G$ is a bipartite graph $T$, where one bipartition set contains the cutvertices of $G$ and the other bipartition set consists of the blocks of~$G$.
For a cutvertex $v$ in $G$ and a block $B$ we have $vB\in E(T)$ if and only if $v$ is contained in~$B$.
If~$G$ is connected, then the block graph of $G$ is a tree,~\mbox{\cf~\cite{Diestel17}}.
Note that the blocks and the block graph can be computed in linear time, \cf~\cite{hopcroft:73}.
We call a block \emph{endpoint} if it is a leaf in $T$.
Further, we call the unique cutvertex in an endpoint its \emph{connection vertex}.

A subset $D \subseteq V(G)$ \emph{dominates} a vertex $v$ if $v\in D$ or $v$ is adjacent to a vertex in $D$.
The set $D$ is a \emph{dominating set} of $G$ if every vertex of $G$ is contained in~$D$ or dominated by a vertex in $D$.
Further, a subset $C\subseteq V(G)$ is a \emph{vertex cover} of $G$ if every edge has an endvertex in $C$.
An alternative characterization for a vertex cover $C$ in a graph is:
\begin{align}
	C \text{ is a vertex cover in } G \Leftrightarrow \forall v\in V(G)\colon v\in C \text{ or } N(v)\in C.\label{vc}
\end{align}

Further, we use standard notation and basic results from linear and integer programming.
For a further introduction into this topic we refer to \cite{grotschel:88}.

\section{Characterization and Complexity of \sdsprob}\label{complexity}

In this section we introduce the basic definitions of simultaneous domination.
Afterwards, we provide an alternative characterization for a \sdset and analyze the complexity of the related decision problem.

\begin{definition}\label{def:sds}
Let $G$ be a connected graph and $S \subseteq V(G)$.
We call $S$ a \emph{\sdset} or \emph{\SDset} of $G$ if~$S$ is a dominating set in every spanning tree of $G$.
A vertex~$v\in V(G)$ is called \emph{\simdomed} by~$S$ if $v$ is dominated by $S$ in every spanning tree of~$G$.
Similarly, a subset~$V^\prime\subseteq V(G)$ is \simdomed by~$S$ if every vertex in $V^\prime$ is \simdomed by~$S$.
\end{definition}

During this article we mainly investigate \emph{\minsdsprob} which consists of finding a \sdset of minimum size.
The decision version of this problem is defined as follows:
\begin{decisionproblem}
	\problemtitle{\sdsprob}
	\probleminput{A connected graph $G$ and an integer $B\in \mathbb{N}$.}
	\problemquestion{Is there a subset $S\subseteq V(G)$ with $|S|\leq B$ such that $S$ is a \sdset in $G$?}
	\problemcomplexity{}
\end{decisionproblem}

Initially, it is not clear if \sdsprob is contained in \NP.
As a graph can have an exponential number of spanning trees we cannot simply test dominance of a given solution in every spanning tree.
However, the following theorem enables us to verify if a given set $S$ is a \sdset in polynomial time.

\begin{theorem}\label{sdscond}
	Let $G$ be a connected graph. A set $S\subseteq V(G)$ is a \sdset if and only if for every $v\in V(G)$ it holds true that $v\in S$ or:
	\begin{thmenumnonum}
		\item $v$ is not a cutvertex in $G$ and $N(v)\subseteq S$, or\label{sdscond_i}
		\item $v$ is a cutvertex in $G$ that is contained in the blocks $B_1,\ldots,B_k$ and for some $i\in \{1,\ldots,k\}$ we have $N_{B_i}(v)\subseteq S$.\label{sdscond_ii}
	\end{thmenumnonum}
\end{theorem}
\begin{proof}
	By Assumption~\ref{ass:conn_2} any connected graph contains at least two vertices. Thus, the neighborhood of a vertex in a connected graph is never empty.
	Let $v \in V(G)\setminus S$ be a vertex that is not a cutvertex in $G$.
	We claim that the vertex~$v$ is \simdomed by $S$ if and only if all neighbors of~$v$ are in $S$:

	If all neighbors of $v$ are contained in $S$, then $v$ is clearly dominated by~$S$ in every spanning tree of~$G$ since there is at least one edge between~$v$ and one of its neighbors in every spanning tree.
	Conversely, assume that $v$ is \simdomed by $S$.
	Since $G-v$ is connected there is a spanning tree of~$G-v$.
	We obtain a spanning tree of $G$ by adding $v$ and any edge incident to $v$ in~$G$.
	Thus, for any neighbor $u$ of $v$ in $G$ there is at least one spanning tree of $G$ such that $u$ is the only neighbor of~$v$.
	Since~$v$ is dominated in every spanning tree of~$G$ and $v\notin S$ we get that all neighbors of $v$ are contained in~$S$.

	Next consider the case that $v$ is a cutvertex and is further contained in the blocks $B_1,\ldots,B_k$.
	We show that $v$ is \simdomed by $S$ if and only if there is an~$i\in\{1,\ldots,k\}$ such that~${w\in S}$ for all $w\in N_{B_i}(v)$:

	If for some~${i\in \{1,\ldots,k\}}$ we have~$w\in S$ for all~${w\in N_{B_i}(v)}$, then $v$ is clearly dominated by $S$ in every spanning tree of $G$ since there is at least one edge between $v$ and one of its neighbors in the block $B_i$.
	Conversely, assume that for each~$i\in\{1,\ldots,k\}$ there is an $w_i\in N_{B_i}(v)$ that is not in $S$.
	We obtain a spanning tree~$T$ of~$G$ by using a spanning forest in~$G-v$ and adding the vertex $v$ and for every $i\in\{1,\ldots,k\}$ the edge~$vw_i$.
	The vertex $v$ is not dominated in $T$ since neither the vertex~$v$ nor any of its neighbors $w_i$ is in~$S$.
	Hence, $S$ is not an \SDset.
\end{proof}

By Theorem~\ref{sdscond} we can verify for a graph $G$ if a given set $S\subseteq V(G)$ is an \SDset in polynomial time by simply checking conditions \ref{sdscond_i} and \ref{sdscond_ii} of Theorem \ref{sdscond} for every vertex $v\in V(G)\setminus S$.
Recall, that for a graph $G$ a set $C\subseteq V(G)$ is a vertex cover if and only if for every vertex $v\in V(G)$ it holds that $v\in C$ or $N_G(v)\subseteq C$, \cf~\eqref{vc}.
Theorem \ref{sdscond} asks for exactly the same for non-cutvertices and hence we get:
\begin{corollary}\label{cor:sds-vc}
	If $G$ is a $2$-connected graph, then $S\subseteq V(G)$ is a \sdset if and only if $S$ is a vertex cover in $G$.\qed
\end{corollary}
\VCprob is one of Karp's 21 \NP-complete problems, \cf~\cite{Garey79}.
It can be observed that the problem remains \NP-hard on $2$-connected graphs and thus:
\begin{corollary}\label{cor:np}
	\sdsprob is \NP-complete.\qed
\end{corollary}
Corollary \ref{cor:sds-vc} reveals a close connection between \sdsprob and \VCprob.
However, it is not immediately clear if and how we may use this relation to efficiently compute a minimum size \sdset in graphs on which \VCprob can be solved in polynomial time.
We later see examples of such possibilities for certain graph classes such as bipartite or chordal graphs, but in the following we proof that this is not always the case. 

Recall the definition of a perfect graph.
A graph $G$ is \emph{perfect} if for every induced subgraph the chromatic number equals the clique number.  The chromatic number is the minimum number of labels needed, such that every vertex has an assigned label and no two adjacent vertices have the same label. The clique number is the size of a largest induced subgraph that is complete.
It is well known that \minVCprob can be solved in polynomial time on perfect graphs, \cf~\cite{schrijver-book}.
However, in the following we prove that \sdsprob is \NP-complete when restricted perfect graphs.

To do so we make use of the Strong Perfect Graph Theorem proven by Chudnovsky \textit{et al.} in \cite{chudnovsky:06}.
Recall that for a graph $G$ an \emph{odd hole} of~$G$ is an induced subgraph of~$G$ which is a cycle of odd length at least $5$. An \emph{odd antihole} of~$G$ is an induced subgraph of~$G$ whose complement is an odd hole in $\bar{G}$.
\begin{theorem}[Strong perfect graph theorem, \cite{chudnovsky:06}]\label{strongperfectgraphtheorem}
	A graph $G$ is perfect if and only if $G$ has no odd hole and no odd antihole. \qed
\end{theorem}

\begin{theorem}\label{thm:sds_np_perfect_graphs}
	\sdsprob is \NP-complete when restricted to perfect graphs.	
\end{theorem}
\begin{proof}
	By Theorem~\ref{sdscond} \sdsprob restricted to perfect graphs is contained in \NP.
	
	It is well known that \VCprob is \NP-complete and it can be observed that it remains \NP-complete on $2$-connected graphs.
	Therefore, let~$G$ be a simple, $2$-connected graph.
	For every edge $uv\in E(G)$ we denote by $H_{uv}$ the graph with
	\begin{align*}
		& V(H_{uv}) = \{u,v,x_1^{uv},x_2^{uv},x_3^{uv},x_4^{uv},y_1^{uv},y_2^{uv},z_1^{uv},z_2^{uv}\} \text{ and}\\
		& E(H_{uv}) = \{ ux_1^{uv}, x_1^{uv}x_2^{uv}, x_2^{uv}x_3^{uv}, x_3^{uv}x_4^{uv}, x_4^{uv}v, , x_1^{uv}y_1^{uv}, y_1^{uv}y_2^{uv}, x_3^{uv}z_1^{uv}, z_1^{uv}z_2^{uv} \}.
	\end{align*}
	The graph $H_{uv}$ is illustrated in Figure~\ref{tikz:sds_np_perfect_graph_G_help_construction}.
	\begin{figure}[htbp]
		\centering
		\begin{minipage}{.45\textwidth}
			\centering
			\begin{tikzpicture}
				\draw[inner sep=1pt]
				(0,0) node[circle,draw,minimum size=20pt,fill=\colorone!60] (u) {\small$u$}
				(1,0) node[circle,draw,minimum size=15pt] (p1) {\scriptsize$x_1^{uv}$}
				(2,0) node[circle,draw,minimum size=15pt,fill=\colorone!60] (p2) {\scriptsize$x_2^{uv}$}
				(3,0) node[circle,draw,minimum size=15pt] (p3) {\scriptsize$x_3^{uv}$}
				(4,0) node[circle,draw,minimum size=15pt,fill=\colorone!60] (p4) {\scriptsize$x_4^{uv}$}
				(5,0) node[circle,draw,minimum size=20pt] (v) {\small$v$}
				(1,1) node[circle,draw,minimum size=15pt,fill=\colorone!60] (y1) {\scriptsize$y_1^{uv}$}
				(1,2) node[circle,draw,minimum size=15pt] (y2) {\scriptsize$y_2^{uv}$}
				(3,1) node[circle,draw,minimum size=15pt,fill=\colorone!60] (z1) {\scriptsize$z_1^{uv}$}
				(3,2) node[circle,draw,minimum size=15pt] (z2) {\scriptsize$z_2^{uv}$}
				;
				
				\draw[-]
				(u) edge (p1)
				(p1) edge (p2)
				(p2) edge (p3)
				(p3) edge (p4)
				(p4) edge (v)
				(p1) edge (y1)
				(y1) edge (y2)
				(p3) edge (z1)
				(z1) edge (z2)
				(p1) edge[bend left=35] (p3)
				;
				
			\end{tikzpicture}
		\end{minipage}
		\begin{minipage}{.45\textwidth}
			\centering
			\begin{tikzpicture}
				\draw[inner sep=1pt]
				(0,0) node[circle,draw,minimum size=20pt] (u) {\small$u$}
				(1,0) node[circle,draw,minimum size=15pt,fill=\colorone!60] (p1) {\scriptsize$x_1^{uv}$}
				(2,0) node[circle,draw,minimum size=15pt] (p2) {\scriptsize$x_2^{uv}$}
				(3,0) node[circle,draw,minimum size=15pt,fill=\colorone!60] (p3) {\scriptsize$x_3^{uv}$}
				(4,0) node[circle,draw,minimum size=15pt] (p4) {\scriptsize$x_4^{uv}$}
				(5,0) node[circle,draw,minimum size=20pt,fill=\colorone!60] (v) {\small$v$}
				(1,1) node[circle,draw,minimum size=15pt,fill=\colorone!60] (y1) {\scriptsize$y_1^{uv}$}
				(1,2) node[circle,draw,minimum size=15pt] (y2) {\scriptsize$y_2^{uv}$}
				(3,1) node[circle,draw,minimum size=15pt,fill=\colorone!60] (z1) {\scriptsize$z_1^{uv}$}
				(3,2) node[circle,draw,minimum size=15pt] (z2) {\scriptsize$z_2^{uv}$}
				;
				
				\draw[-]
				(u) edge (p1)
				(p1) edge (p2)
				(p2) edge (p3)
				(p3) edge (p4)
				(p4) edge (v)
				(p1) edge (y1)
				(y1) edge (y2)
				(p3) edge (z1)
				(z1) edge (z2)
				(p1) edge[bend left=35] (p3)
				;
				
			\end{tikzpicture}
		\end{minipage}
		
		\caption{Construction of two \SDsets in $H_{uv}$.}
		\label{tikz:sds_np_perfect_graph_G_help_construction}
	\end{figure}
	
	\noindent We regard the graph
	\begin{align*}
		H=\bigcup_{uv\in E(G)} H_{uv}
	\end{align*}
	and claim that $H$ is perfect.
	Further, we claim that $G$ has a vertex cover of size at most $k$ if and only if $H$ has an \SDset of size at most $k+|E(G)|$.
	
	To show that the graph $H$ is perfect we use the Strong Perfect Graph Theorem, \cf~Theorem~\ref{strongperfectgraphtheorem}, and show that there is no odd hole nor an odd antihole.
	First note that the only vertices in $H$ that can possibly have degree larger than $5$ are the ones also contained in $G$.
	As none of these are adjacent in $H$ there cannot be an odd antihole of size $7$ or larger.
	Further, the only cycle completely contained inside of some $H_{uv}$ for $uv\in E(G)$ is of the form $x_1^{uv}x_2^{uv}x_3^{uv}x_1^{uv}$ and has length $3$.
	Any other cordless cycle $C$ in $H$ that passes through $H_{uv}$ has to use the path $ux_1^{uv}x_3^{uv}x_4^{uv}v$ as otherwise it contains the chord $x_1^{uv}x_3^{uv}$.
	However, this path has length $4$ and since this holds for every $H_{uv}$ we get that $C$ has even length and there is no odd hole in $H$.
	Since an odd anti-hole of size $5$ is the same as an odd hole of size $5$ it already follows by the Strong Perfect Graph Theorem that $H$ is perfect.
	
	Next we show that $G$ has a vertex cover of size at most $k$ if and only if $H$ has an \SDset of size at most $k+4|E(G)|$.
	
	Let $C\subseteq V(G)$ be a vertex cover of $G$ of size $k$.
	Consider the set $S$ that contains all vertices from $C$ and for every $uv\in E(G)$ the vertices $x_1^{uv},x_3^{uv},y_1^{uv},z_1^{uv}$ if $u\notin C$ and $x_2^{uv},x_4^{uv},y_1^{uv},z_1^{uv}$ if $u\in C$, \cf~Figure~\ref{tikz:sds_np_perfect_graph_G_help_construction}.
	Clearly, $|S|=k+4|E(G)|$ and we claim that~$S$ is a \sdset of $H$.
	To this end regard some $uv\in E(G)$.
	Since $y_1^{uv}$ and $z_1^{uv}$ are in $S$ and $x_1^{uv}$ and $x_3^{uv}$ are cutvertices in $H$ Theorem~\ref{sdscond} implies that $x_1^{uv}$ and $x_3^{uv}$ are \simdomed.
	For $w\in\{ x_2^{uv},x_4^{uv},y_1^{uv},y_2^{uv},z_1^{uv},z_2^{uv}\}$ we have that either $w$ itself is in $S$ or all neighbors of $w$ are in $S$.
	Thus, by Theorem~\ref{sdscond} all vertices in $H_{uv}$ except possibly $u$ and $v$ are \simdomed.
	If $u\notin S$, then by definition of $S$ all neighbors of $u$ in $H$ are contained in $S$ and $u$ is \simdomed.
	If $v\notin S$, then it must be the case that $u\in S$ as $C\subset S$ is a vertex cover of~$G$.
	Again it follows by the definition of $S$ that all neighbors of $v$ are contained in $S$ and $w$ is \simdomed by $S$.
	Overall, we conclude that $S$ is a \sdset in $H$.
	
	Now let $S\subseteq V(H)$ be an \SDset in $H$ of size $k+|E(G)|$.
	First we show that for each $uv\in E(G)$ we have $|S\cap (V(H_{uv})\setminus\{u,v\})|\geq 4$.
	To this end note, that
	\begin{align}
		|S\cap\{y_1^{uv},y_2^{uv},z_1^{uv},z_2^{uv}\}|\geq 2. \label{eq:geq2}
	\end{align}
	Further, $x_2^{uv}$ and $x_4^{uv}$ are not cutvertices in $H$ as $G$ is $2$-connected.
	Thus, by Theorem~\ref{sdscond}, we either have $x_3^{uv}\notin S$ in which case $x_2^{uv},x_4^{uv}\in S$, or $x_3^{uv}\in S$ in which case $x_1^{uv}\in S$ or $x_{2}^{uv}\in S$.
	We conclude that for each $uv\in E(G)$ we have
	\begin{align}
		|S\cap (V(H_{uv})\setminus \{u,v\})|\geq 4.\label{eq:geq4}
	\end{align}
	As $G$ is $2$-connected we have for every edge $uv\in E(G)$ that neither $u$ nor $v$ is a cutvertex in $H$.
	If both, $u$ and $v$, are not contained in $S$, we have by Theorem~\ref{sdscond} that $x_1^{uv}$ as well as $x_{4}^{uv}$ are in $S$.
	Further, as $x_2^{uv}$ is not a cutvertex in $H$ we have $x_2^{uv}\in S$ or $x_3^{uv}\in S$.
	By \eqref{eq:geq2} it follows that $|S\cap V(H_{uv})|\geq 5$.
	Replacing the elements in $S\cap V(H_{uv})$ by the elements in the set $\{ u, x_2^{uv}, x_4^{uv}, y_1^{uv},z_1^{uv} \}$ yields an \SDset of no larger cardinality which contains $v$.
	Thus we may assume that for every $uv\in E(G)$ we have $u\in S$ or $v\in S$.
	Therefore, the set $C=S\cap V(G)$ is a vertex cover and by \eqref{eq:geq4} $|C|\leq k$.
\end{proof}

Theorem~\ref{thm:sds_np_perfect_graphs} demonstrates that \sdsprob and \VCprob differ in their complexity on some graph classes, unless $\Pp\neq\NP$.
However, in the following we show that the gap between a minimum size \sdset and a minimum size vertex cover cannot be too large.
In particular, we demonstrate that a minimum size vertex cover may be at most twice as large as a minimum size \SDset.
For an integer $k\in \mathbb{N}$ let $G$ be a graph on $3k$ vertices, where $k$ vertices form a clique and each vertex of the clique has a dangling path of length two attached to it, \cf~Figure~\ref{tikz::example_sds_vs_vc_VC}.
\begin{figure}[ht]
	
	\hfill
	\begin{minipage}[t]{.47\textwidth}%
		\centering
		\begin{tikzpicture}
			\draw[inner sep=1pt, minimum size = 10pt]
			(0,1) node[circle,draw] (v1) {}
			(0.71,0.71) node[circle,draw,fill=\colorone!60] (v2) {}
			(1,0) node[circle,draw,fill=\colorone!60] (v3) {}
			(0.71,-0.71) node[circle,draw,fill=\colorone!60] (v4) {}
			(0,-1) node[circle,draw,fill=\colorone!60] (v5) {}
			(-0.71,-0.71) node[circle,draw,fill=\colorone!60] (v6) {}
			(-1,0) node[circle] (v7) {$\vdots$}
			(-0.71,0.71) node[circle,draw,fill=\colorone!60] (v8) {}
			
			(0,1.75) node[circle,draw,fill=\colorone!60] (u1) {}
			(1.24,1.24) node[circle,draw,fill=\colorone!60] (u2) {}
			(1.75,0) node[circle,draw,fill=\colorone!60] (u3) {}
			(1.24,-1.24) node[circle,draw,fill=\colorone!60] (u4) {}
			(0,-1.75) node[circle,draw,fill=\colorone!60] (u5) {}
			(-1.24,-1.24) node[circle,draw,fill=\colorone!60] (u6) {}
			(-1.24,1.24) node[circle,draw,fill=\colorone!60] (u8) {}
			
			(0,2.5) node[circle,draw] (w1) {}
			(1.78,1.78) node[circle,draw] (w2) {}
			(2.5,0) node[circle,draw] (w3) {}
			(1.78,-1.78) node[circle,draw] (w4) {}
			(0,-2.5) node[circle,draw] (w5) {}
			(-1.78,-1.78) node[circle,draw] (w6) {}
			(-1.78,1.78) node[circle,draw] (w8) {}
			;

			\draw[-, thick]
			(v1) edge (v2)
			(v1) edge (v3)
			(v1) edge (v4)
			(v1) edge (v5)
			(v1) edge (v6)
			(v1) edge (v8)
			(v2) edge (v3)
			(v2) edge (v4)
			(v2) edge (v5)
			(v2) edge (v6)
			(v2) edge (v8)
			(v3) edge (v4)
			(v3) edge (v5)
			(v3) edge (v6)
			(v3) edge (v8)
			(v4) edge (v5)
			(v4) edge (v6)
			(v4) edge (v8)
			(v5) edge (v6)
			(v5) edge (v8)
			(v6) edge (v8)
			(v7) edge (v1)
			(v7) edge (v2)
			(v7) edge (v3)
			(v7) edge (v4)
			(v7) edge (v5)
			(v7) edge (v6)
			(v7) edge (v8)
			
			(v1) edge (u1)
			(v2) edge (u2)
			(v3) edge (u3)
			(v4) edge (u4)
			(v5) edge (u5)
			(v6) edge (u6)
			(v8) edge (u8)
			
			(w1) edge (u1)
			(w2) edge (u2)
			(w3) edge (u3)
			(w4) edge (u4)
			(w5) edge (u5)
			(w6) edge (u6)
			(w8) edge (u8)
			;
		\end{tikzpicture}
		\caption{Vertex cover of minimum size in $G$}
		\label{tikz::example_sds_vs_vc_VC}
	\end{minipage}
	\hfill
	\begin{minipage}[t]{.47\textwidth}%
		\centering
		\begin{tikzpicture}
			\draw[inner sep=1pt, minimum size = 10pt]
			(0,1) node[circle,draw] (v1) {}
			(0.71,0.71) node[circle,draw] (v2) {}
			(1,0) node[circle,draw] (v3) {}
			(0.71,-0.71) node[circle,draw] (v4) {}
			(0,-1) node[circle,draw] (v5) {}
			(-0.71,-0.71) node[circle,draw] (v6) {}
			(-1,0) node[circle] (v7) {$\vdots$}
			(-0.71,0.71) node[circle,draw] (v8) {}
			
			(0,1.75) node[circle,draw,fill=\colorone!60] (u1) {}
			(1.24,1.24) node[circle,draw,fill=\colorone!60] (u2) {}
			(1.75,0) node[circle,draw,fill=\colorone!60] (u3) {}
			(1.24,-1.24) node[circle,draw,fill=\colorone!60] (u4) {}
			(0,-1.75) node[circle,draw,fill=\colorone!60] (u5) {}
			(-1.24,-1.24) node[circle,draw,fill=\colorone!60] (u6) {}
			(-1.24,1.24) node[circle,draw,fill=\colorone!60] (u8) {}
			
			(0,2.5) node[circle,draw] (w1) {}
			(1.78,1.78) node[circle,draw] (w2) {}
			(2.5,0) node[circle,draw] (w3) {}
			(1.78,-1.78) node[circle,draw] (w4) {}
			(0,-2.5) node[circle,draw] (w5) {}
			(-1.78,-1.78) node[circle,draw] (w6) {}
			(-1.78,1.78) node[circle,draw] (w8) {}
			;

			\draw[-, thick]
			(v1) edge (v2)
			(v1) edge (v3)
			(v1) edge (v4)
			(v1) edge (v5)
			(v1) edge (v6)
			(v1) edge (v8)
			(v2) edge (v3)
			(v2) edge (v4)
			(v2) edge (v5)
			(v2) edge (v6)
			(v2) edge (v8)
			(v3) edge (v4)
			(v3) edge (v5)
			(v3) edge (v6)
			(v3) edge (v8)
			(v4) edge (v5)
			(v4) edge (v6)
			(v4) edge (v8)
			(v5) edge (v6)
			(v5) edge (v8)
			(v6) edge (v8)
			(v7) edge (v1)
			(v7) edge (v2)
			(v7) edge (v3)
			(v7) edge (v4)
			(v7) edge (v5)
			(v7) edge (v6)
			(v7) edge (v8)
			
			(v1) edge (u1)
			(v2) edge (u2)
			(v3) edge (u3)
			(v4) edge (u4)
			(v5) edge (u5)
			(v6) edge (u6)
			(v8) edge (u8)
			
			(w1) edge (u1)
			(w2) edge (u2)
			(w3) edge (u3)
			(w4) edge (u4)
			(w5) edge (u5)
			(w6) edge (u6)
			(w8) edge (u8)
			;
		\end{tikzpicture}
		\caption{\SDset of minimum size in~$G$}
		\label{tikz::example_sds_vs_vc_SDS}
	\end{minipage}
	\hfill
\end{figure}
A vertex cover in $G$ contains at least $k-1$ vertices of the clique as well as one more for each dangling path.
Hence, a vertex cover of $G$ has size at least~$2k-1$.
See Figure \ref{tikz::example_sds_vs_vc_VC} for a possible minimum size vertex cover.

On the other hand there is an \SDset of size $k$ since the edges of every dangling path are contained in a spanning tree of $G$, \cf~Figure~\ref{tikz::example_sds_vs_vc_SDS}.
As for every dangling path at least one of its vertices has to be contained in an \sdset the described \SDset is of minimum size.
Now we show that this is the largest gap possible.
\begin{theorem}\label{sds_spts:all_spanning_trees:thm:2timesVC}
	Let $G$ be a connected graph and $S$ an \SDset.
	We can extend $S$ to a vertex cover C by adding at most $|S|-1$ vertices.
	In particular, if $C'$ is a minimum size vertex cover and $S'$ is a minimum size \SDset, then it holds that $|C'|\leq 2|S'|-1$. The given bound is tight.
\end{theorem}
\begin{proof}
	Let $G$ be a connected graph and denote the block graph of $G$ by $T$.
	Further, let $S$ be an \SDset.
	By definition there is at least one vertex in $S$ and thus, there is a block $B$ such that $S\cap V(B)\neq\emptyset$.
	We root $T$ at such a block $B_r$.
	Denote by $\CV$ the set of cutvertices of~$G$.
	For every cutvertex with children $B_1,\ldots,B_k$ in $T$ let 
	\begin{align*}
		S(v)=S\cap \left( \bigcup_{i=1}^k N_{B_i}(v) \right),
	\end{align*}
	\ie, $S(v)$ consists of the neighbors of $v$ that are in $S$ and in the blocks that are children of $v$ in $T$.
	We claim that $C$ is a vertex cover of size at most $2|S|-1$, where
	\begin{align*}
		C\coloneqq S\cup\{ v\in \CV\colon S(v)\neq \emptyset \}.
	\end{align*}
	An example for such a set $C$ is illustrated in Figure~\ref{tikz:constrof_S(v)}.
	
	\begin{figure}[htb]
		\centering
		\begin{minipage}{.45\textwidth}
			\centering
			\begin{tikzpicture}
				\draw[inner sep=1pt, minimum size = 20pt]
				(0,0) node[draw,circle,fill=\colortwo!60] (v) {\scriptsize$v_1$}
				;
				\draw[rectangle,minimum size=80,rounded corners=25pt]
				(-1.5,-3) node[draw,fill=black!10] (b1) {}
				(1.5,-3) node[draw,fill=black!10] (b2) {}
				;
				\draw[-]
				(v) edge ([shift={(0,40pt)}]b1.center)
				(v) edge ([shift={(0,40pt)}]b2.center)
				;
				\draw[inner sep=1pt, minimum size = 15pt,circle]
				(-1.5,-2.3) node[draw,fill=\colortwo!60] (v1) {\scriptsize$v_1$}
				(-2.4,-2.735) node[draw,fill=\colorone!60] (u1) {\scriptsize$u_1$}
				(-1.5,-3.17) node[draw,fill=white] (u2) {\scriptsize$u_2$}
				(-0.6,-2.735) node[draw,fill=\colorone!60] (u3) {\scriptsize$u_3$}
				;
				\draw[dashed,color=black]
				(v1) edge (u1)
				(u1) edge (u2)
				(u2) edge (u3)
				(u3) edge (v1)
				;
				\draw
				(-2.3,-3.95) node (b11) {\large$B^{v_1}_1$}
				(.7,-3.95) node (b21) {\large$B^{v_1}_2$}
				;
				\draw[inner sep=1pt, minimum size = 15pt,circle]
				(1.5,-2.3) node[draw,fill=\colortwo!60] (v2) {\scriptsize$v_1$}
				(1.5,-3.17) node[draw,fill=\colortwo!60] (u4) {\scriptsize$u_4$}
				;
				\draw[dashed,color=black]
				(v2) edge (u4)
				;
				\draw[dotted,color=black,thick]
				(v) edge ([shift={(0,23pt)}]v) 
				(b1) edge ([shift={(-1,-52pt)}]b1) 
				(b1) edge ([shift={(1,-52pt)}]b1)
				(b2) edge ([shift={(0,-52pt)}]b2) 
				;
			\end{tikzpicture}
		\end{minipage}
		\begin{minipage}{.45\textwidth}
			\centering
			\begin{tikzpicture}
				\draw[inner sep=1pt, minimum size = 20pt]
				(0,0) node[draw,circle] (v) {\scriptsize$v_2$}
				;
				\draw[rectangle,minimum size=80,rounded corners=25pt]
				(-1.5,-3) node[draw,fill=black!10] (b1) {}
				(1.5,-3) node[draw,fill=black!10] (b2) {}
				;
				\draw[-]
				(v) edge ([shift={(0,40pt)}]b1.center)
				(v) edge ([shift={(0,40pt)}]b2.center)
				;
				\draw[inner sep=1pt, minimum size = 15pt,circle]
				(1.5,-2.3) node[draw,fill=white] (v1) {\scriptsize$v_2$}
				(2,-3.17) node[draw,fill=\colortwo!60] (u1) {\scriptsize$w_3$}
				(1,-3.17) node[draw,fill=\colortwo!60] (u2) {\scriptsize$w_2$}
				;
				\draw[dashed,color=black]
				(v1) edge (u1)
				(v1) edge (u2)
				(u1) edge (u2)
				;
				\draw
				(-2.3,-3.95) node (b11) {\large$B^{v_2}_1$}
				(.7,-3.95) node (b21) {\large$B^{v_2}_2$}
				;
				\draw[inner sep=1pt, minimum size = 15pt,circle]
				(-1.5,-2.3) node[draw,fill=white] (v2) {\scriptsize$v_2$}
				(-1.5,-3.17) node[draw,fill=\colortwo!60] (u3) {\scriptsize$w_1$}
				;
				\draw[dashed,color=black]
				(v2) edge (u3)
				;
				\draw[dotted,color=black,thick]
				(v) edge ([shift={(0,23pt)}]v) 
				(b2) edge ([shift={(-1,-52pt)}]b2) 
				(b2) edge ([shift={(1,-52pt)}]b2) 
				(b1) edge ([shift={(0,-52pt)}]b1) 
				;
			\end{tikzpicture}
		\end{minipage}
		\caption{Part of the block graph of a graph $G$: One can observe the cutvertices $v_1$ and $v_2$ together with their children.
			The solid and dotted edges are edges in the block graph.
			The dashed edges are edges in the graph $G$.
			The \coloronetext vertices are in $S$ and the \colortwotext vertices are in $C\setminus S$.}
		\label{tikz:constrof_S(v)}
	\end{figure}
	
	First we prove that $C$ is in fact a vertex cover.
	To do so, we show for every vertex $v\in V(G)\setminus C$ that $N_G(v)\subseteq C$.
	First of all, if $v$ is not a cutvertex, then all its neighbors in $G$ are in $S$ by condition \ref{sdscond_i} of Theorem \ref{sdscond} and thus in $C$.
	Next assume that $v$ is a cutvertex with children $B_1,\ldots,B_k$ and parent $B$ in $T$.
	Since $v\notin C$ we have $S(v)=\emptyset$ and thus no neighbor of $v$ in the children $B_1,\ldots,B_k$ is in $S$.
	Condition~\ref{sdscond_ii} of Theorem~\ref{sdscond} implies that all neighbors of $v$ in $B$ are in~$S$ and hence in $C$.
	Now consider a vertex $w\in N(v)\cap B_i$ for $i\in \{1,\ldots,k\}$.
	If we show that $w\in C$, then the claim follows.
	Since neither $v$ nor $w$ is in $S$ the vertex $w$ needs to be a cutvertex, otherwise $w$ would not satisfy condition~\ref{sdscond_i} of Theorem~\ref{sdscond}.
	Further, the block $B_i$ is the parent of $w$ in $T$ and since $v\notin S$ we have that $w$ is \simdomed in a block that is its child in $T$.
	In particular, we have $S(w)\neq \emptyset$ and by the definition of $C$ it is $w\in C$.
	Overall, this shows that $N_G(v)\subseteq C$ and hence $C$ is a vertex cover.
	
	
	Next we show $|C|\leq 2|S|-1$.
	This follows if we can find an injective mapping from $C\setminus S$ to $S\setminus V(B_r)$ since $|S\cap V(B_r)|\geq 1$.
	By the definition of $C$ there is for every $v\in C\setminus S$ a vertex $w\in S(v)\subseteq S\setminus V(B_r)$.
	If we map $v$ to $w$, then we obtain an injective mapping:
	If $w$ is not a cutvertex, then $w$ is contained in exactly one block and $v$ is the unique parent of this block.
	Otherwise, if $w$ is a cutvertex, then it is itself a child of the block containing $v$ and $w$ in $T$.
	Hence, every block containing $w$ is either a child of $v$ or a child of $w$ and since $w\in S$ no other vertex in $C\setminus S$ is mapped to $w$.
	This shows that the mapping is injective and hence $|C|\leq 2|S|-1$.
	
	We already observed before the statement of this theorem, that there exists a graph $G$ with minimum vertex cover $C$ and minimum \SDset~$S$ such that $|C|=2|S|-1$, \cf~Figure~\ref{tikz:constrof_S(v)}. Thus, the provided bound is in fact tight.
\end{proof}

\section{An Exact Algorithm for \sdsprob using an Oracle for \VCprob}
\label{sec::sds_alg}

In the previous section we saw that on $2$-connected graphs \sdsprob is equivalent to \VCprob.
However, we have also highlighted some differences.
On the one hand we showed that in general graphs the size of a minimum size \SDset and a minimum size vertex cover may differ by a factor of two.
On the other hand we proved that \sdsprob is \NP-complete when restricted to perfect graphs whereas \VCprob is solvable in polynomial time.
In this section we concentrate on the algorithmic aspect of \sdsprob.
In particular, we show how we can find a minimum size \SDset in general graphs using an oracle for a minimum size vertex cover.
To this end we need some further notation.

In the following we assign colors to vertices.
To get an intuition what these colors represent for a vertex $v$ we give an interpretation of them:
\begin{itemize}
	\setlength{\itemsep}{1pt}
	\setlength{\parskip}{0pt}
	\setlength{\parsep}{0pt}
	\item color $\instate$ indicates that $v$ is fixed to be in the \SDset,
	\item color $\domstate$ indicates that $v$ is not in the \SDset yet but it is \simdomed and
	\item color $\notdomstate$ indicates that $v$ is not in the \SDset and it is not \simdomed yet.
\end{itemize}
We call color $\instate$ \emph{better} than color $\domstate$ and $\notdomstate$ and say that color $\domstate$ is \emph{better} than color $\notdomstate$.
For a subset $\operatorname{col}\subseteq \{\instate,\domstate,\notdomstate\}$ we denote the best color of $\operatorname{col}$ by $\operatorname{best}(\operatorname{col})$.

Now we briefly describe the idea of the algorithm:
\begin{idea}\label{idea}
	Let $G$ be a graph in which we want to compute a minimum size \SDset.
	Our algorithm is based on the structure of the block graph of $G$.
	We start with an endpoint $H$ and its connection vertex $v$ of the graph $G$.
	We take out of all minimum size sets $S\subseteq V(H)$ that \simdom all vertices in $V(H)\setminus \{v\}$ one with the \emph{best coverage} for $v$, \ie, the best color for $v$.
	We then remove $H-v$ from~$G$ and continue with the next endpoint.
	
	In later stages of the algorithm we may have vertices in our endpoint, that are already \simdomed or even contained in an \SDset for \emph{free}.
	This has to be taken into account when computing such a minimum size set of a block that was originally not an endpoint.
	The crucial point of this procedure is, that any vertex can be \simdomed by adding only one vertex, namely the vertex itself.
	Thus, if the connection vertex $v$ is not \simdomed in one of its endpoints by any of the minimum size \SDsets, then we can simply \simdom it later on in a subsequent step of the algorithm.
	This is true, as we can be sure that it never costs us more than it would cost us to \simdom it within the current block.
\end{idea}

To formalize this setup where some vertices are already \simdomed or even in the \SDset we need a generalized version of \sdsets.
\begin{definition}\label{def:colresp}
	Let $G$ be a connected graph and $f\colon V(G)\rightarrow \{\instate,\domstate,\notdomstate\}$ a coloring.
	We call a subset $S\subseteq V(G)$ an \emph{$f$-\rsds} if the following conditions hold:
	\begin{itemize}
		\item $f^{-1}(\instate)\subseteq S$ and
		\item $f^{-1}(\notdomstate)$ is \simdomed by $S$.
	\end{itemize}
	If we do not specify the coloring, then we also use the term \emph{color \rsds}.
\end{definition}
Thus, a color \rSDS $S$ is an \SDset such that all vertices with assigned color~$\instate$ are contained in $S$ and all vertices with assigned color~$\domstate$ do not have to be \simdomed by $S$.
Clearly this is a generalization of an \SDset as if all vertices are assigned color~$\notdomstate$, then a color \rSDS and an \SDset are the same thing.

In the following we present an algorithm that computes a minimum size color \rSDS in a connected graph without a cutvertex and afterwards, we show how we can use this algorithm to obtain an \SDset in general graph.
To this end let $G$ be a connected graph without a cutvertex and $f\colon V(G)\rightarrow \{\instate,\domstate,\notdomstate\}$ a coloring of the vertices of $G$.
Algorithm \ref{alg::SDS_gen} describes how to find a minimum size $f$-\rSDS in $G$ using an oracle \textsc{MinVertexCover} for solving the well known \minVCprob as a black box algorithm.
Before we use the oracle to obtain such a vertex cover, we modify the graph.
Recall that all vertices with color $\instate$ have to be in $S$ and therefore we remove them from $G$.
The vertices with color $\domstate$ do not have to be simultaneously dominated and hence we remove the edges between vertices with color $\domstate$.
In Theorem~\ref{thm::second_algo} we prove that a minimum size vertex cover in the modified graph is in fact a $f$-\rSDS in the original graph.
\begin{figure*}[ht]
	\begin{algorithm}[H]
		\caption{$\textsc{crSDS}(G, f)$: Finding a minimum size color \rsds in a connected graph $G$ without a cutvertex}\label{alg::SDS_gen}
		\KwIn{A connected graph $G$ without a cutvertex and a colouring $f\colon V(G)\to\{\instate, \domstate, \notdomstate\}$}
		\KwOut{A minimum $f$-\rsds and its size}
		$G=G-f^{-1}(\instate)$
		
		$G=G- E(G[f^{-1}(\domstate)])$
		
		$S=\textsc{MinVertexCover}(G)$ \label{line_alg:linevcover}
		
		\Return $S\cup f^{-1}(\instate)$, $|S\cup f^{-1}(\instate)|$
	\end{algorithm}
\end{figure*}
\begin{samepage}
	\begin{theorem}\label{thm::second_algo}
		Given a connected graph $G$ without a cutvertex and a coloring $f\colon V(G)\to\{\instate,\domstate,\notdomstate\}$ Algorithm~\ref{alg::SDS_gen} returns a minimum size $f$-\rsds.
		It can be implemented to run in polynomial time if {\normalfont\textsc{MinVertexCover}} can be implemented to run in polynomial time.
	\end{theorem}
\end{samepage}
Before we start with the proof note that this running time is also called \emph{oracle polynomial} given the oracle \textsc{MinVertexCover}.
As oracle algorithms are not our focus and we just use it here we do not formally introduce this form of algorithms.
\begin{proof}
	Let $S^\star$ be the set returned by the algorithm.
	We begin by proving that $S^\star$ is an $f$-\rsds.
	Clearly $f^{-1}(\instate)\subseteq S^\star$ and thus, as $G$ does not contain a cutvertex and by Definition~\ref{def:colresp} we only need to prove that for all vertices~$v$ with $f(v)=\notdomstate$, we have $v\in S^\star$ or $N_G(v)\subseteq S^\star$.
	So let $v\in V(G)$ with $f(v)=\notdomstate$.
	After having deleted all vertices with color $\instate$ we do not delete edges incident to $v$.
	Thus, the vertex cover computed either contains $v$ itself or all neighbors of $v$ which do not have color $\instate$. 
	As all deleted vertices are contained in $S^\star$ the required condition follows and we conclude that $S^\star$ is indeed an $f$-\rsds.
	
	Let $G^\prime=(G-f^{-1}(\instate))-E(G[f^{-1}(\domstate)])$.
	To see that the algorithm returns a minimum size $f$-\rsds we show that for every $f$-\rsds $S$ it holds that $S\setminus f^{-1}(\instate)$ is a vertex cover in $G^\prime$.
	So let~$S$ be any $f$-\rsds and let $e=uv\in E(G^\prime)$.
	Then at least one endpoint of $e$, say $v$, has color $\notdomstate$ and neither $u$ nor $v$ has color~$\instate$.
	By Definition~\ref{def:colresp} this means that~$v$ or all vertices in $N_G(v)$ are contained in $S$.
	We have $u\in N_{G^\prime(v)}$ and thereby~$u\in S\setminus f^{-1}(\instate)$ or $v\in S\setminus f^{-1}(\instate)$. 
	As $e$ was an arbitrary edge in $E(G^\prime)$ we conclude that $S\setminus f^{-1}(\instate)$ is a vertex cover in $G^\prime$.
	
	Clearly all steps of the algorithm, except possibly the call to \textsc{MinVertexCover} can be implemented to run in polynomial time.
\end{proof}

Now that we know how to find a color \rSDS on connected graphs without a cutvertex we focus on minimum size \SDsets in a general graph $G$.
As already described in Explanation~\ref{idea} we make use of the tree structure of the block graph of $G$.
In particular, we do not consider the whole graph $G$ at once but successively work with endpoints of the block graph and their connection vertex.
Since we have to make some adjustments to the used coloring during the algorithm we need further notation to make the arguments more clear and formally correct.

\begin{definition}
	Let $G$ be a graph and $f\colon V(G)\to\{\instate,\domstate,\notdomstate\}$	some coloring of the vertices.
	For any induced subgraph $H$ of $G$ we denote by $f^H$ the coloring $f$ restricted to the nodes of $H$.
	Further, for any fixed vertex $v\in V(H)$ and $i\in\{\instate,\domstate,\notdomstate\}$ we denote by~$f^H_{v=i}$ the coloring of $V(H)$ with $f^H_{v=i}(v)=i$ and~${f^H_{v=i}(w)=f^H(w)}$ for all $w\in V(H)\setminus\{v\}$.
	Finally, we denote by $S^H_{v=i}$ a minimum~${f^H_{v=i}}$-\rsds in $H$ for $i\in\{\instate,\domstate,\notdomstate\}$.
	If $H=G$, then we omit the superscript $H$ in the notation.
\end{definition}

Algorithm \ref{alg::SDS_STF} shows a pseudo code version of the complete procedure.
Within the algorithm we use the algorithm \textsc{crSDS} and the \emph{black box algorithm} \textsc{GetEndPoint}.
The latter one takes as input a graph $G$ that is not $2$-connected and returns an endpoint~$B$ of the block graph of $G$ and the parent $v\in B$ of the endpoint $B$ in the block graph.
Note that if a vertex in $V(B)$ is \simdomed in $B$, then this vertex is \simdomed in $G$ by Theorem~\ref{sdscond}.
Therefore, a color \rsds in $B$ suffices to ensure that every vertex in $V(B)\setminus \{v\}$ with color $\notdomstate$ is \simdomed in $G$.
We save the current color of $v$ and compute a color \rsds in $B$ for every possible color of $v$.
We use the color \rsds in $B$, which is the smallest among the three possibilities, where ties are broken by the best coverage of~$v$.
Afterwards, we delete $B-v$ from $G$ and continue with the remaining graph.
Before we formally prove the correctness of Algorithm~\ref{alg::SDS_STF} and discuss its running time, we prove two lemmas, which make life easier in the proof of the algorithm.

\begin{figure*}[ht]
	\begin{algorithm}[H]
		\caption{Computing a color respecting \SDset of minimum size}\label{alg::SDS_STF}
		\KwIn{A connected graph $G$ and a coloring $f\colon V(G)\to\{\instate, \domstate, \notdomstate\}$}
		\KwOut{An $f$-\rSDS of minimum size in $G$}	
		$S=\emptyset$
		
		\While{$G$ contains a cutvertex}{
			$B, v=\textsc{getEndPoint}(G)$
			
			$c^\star=f(v)$
			
			\For{$c\in\{\instate, \domstate, \notdomstate\}$}{
				$f(v)=c$
				
				$S_c, \#_c=\textsc{crSDS}(B, f^{B})$\label{alg:subproblemcomp}
			}
			$f(v)=c^\star$
			
			\If{$\#_\instate=\#_{\notdomstate}=\#_{\domstate}$}{
				$f(v)=\instate$\label{alg:difcol1}
				
				$S=S\cup S_{\instate}$
			}
			\ElseIf{$\#_\instate>\#_{\notdomstate}=\#_\domstate$}{
				$f(v)=\best({c^\star, \domstate})$\label{alg:difcol2}
				
				$S=S\cup S_{\notdomstate}$
				
			}
			\Else{
				$S=S\cup S_{\domstate}$
			}
			$G= G-\left(V(B)\setminus\{v\}\right)$
		}
		$S_G, \# = \textsc{crSDS}(G, f)$\label{alg:linek1}
		
		\Return $S\cup S_G$
		
	\end{algorithm}
\end{figure*}

Before we formally prove the correctness of Algorithm~\ref{alg::SDS_STF} and discuss its running time, we prove two lemmas, which make life easier in the proof of the algorithm.
Lemma \ref{lem:maxone} shows that if we only change one color in a coloring of a connected graph~$G$ without a cutvertex, then the size of a minimum color \rSDS changes at most by one.
\begin{lemma}\label{lem:maxone}
	Let $G$ be a connected graph without a cutvertex, $v\in V(G)$ a fixed vertex in $G$ and ${f\colon V(G)\to\{\instate,\domstate,\notdomstate\}}$ some coloring. 
	Then the following two statements hold:
	\begin{samepage}
		\begin{thmenum}
			\item $|S_{v=\domstate}|\leq|S_{v={\notdomstate}}|\leq |S_{v=\instate}|$.\label{lem:maxone1}
			\item $|S_{v=\instate}|-|S_{v=\domstate}|\leq 1$.\label{lem:maxone2}
		\end{thmenum}
	\end{samepage}
\end{lemma}
\begin{proof}
	First we show that every $f_{v=\notdomstate}$-\rSDS $S$ is also $f_{v=\domstate}$-respecting. 
	Clearly we have~$f_{v=\domstate}^{-1}(\instate)=f_{v=\notdomstate}^{-1}(\instate)\subseteq S$.
	Further ${f_{v=\domstate}^{-1}(\notdomstate)\subseteq f^{-1}_{v=\notdomstate}(\notdomstate)}$.
	Thus, $f^{-1}_{v=\domstate}(\notdomstate)$ is \simdomed by $S$ and $S$ is also $f_{v=\domstate}$-respecting.
	With similar arguments we get that any $f_{v=\instate}$-\rSDS is also $f_{v=\notdomstate}$-respecting.
	These two small observations directly imply \ref{lem:maxone1}.
	
	To see that \ref{lem:maxone2} holds, let $S_{v=\domstate}$ be a minimum $f_{v=\domstate}$-\rSDS. Then~$S_{v=\domstate}\cup\{v\}$ is $f_{v=\instate}$-respecting, as 
	\begin{align*}
		f_{v=\instate}^{-1}(\instate)=f_{v=\domstate}^{-1}(\instate)\cup\{v\}\subseteq S_{v=\domstate}\cup\{v\}
	\end{align*}
	and $f_{v=\instate}^{-1}(\notdomstate)\subseteq f_{v=\domstate}^{-1}(\notdomstate)$.
	This already implies that the minimum $f_{v=\instate}$-\rSDS has at most one element more than $S_{v=\domstate}$.
\end{proof}
The next lemma justifies how the algorithm combines such a color \rSDS of an endpoint and one of the rest of the graph to obtain a color \rSDS for the whole graph $G$.
For better readability we abuse notation in Lemma~\ref{puzzletogether} and Theorem~\ref{thm::main_algo} and write $G-H$ instead of $G-V(H)$.
As we only use this in these two statements we do not introduce this notation formally.
\begin{lemma}\label{puzzletogether}
	Let $G$ be a graph with some coloring $f\colon V(G)\to\{\instate, \domstate, \notdomstate\}$ and~$H'$ be some endpoint of $G$ with connection vertex $v\in V(G)$ and let $H=H'-v$.
	Then the following three statements hold true:
	\begin{itemize}
		\item If $|S^{H^\prime}_{v=\domstate}|=|S^{H^\prime}_{v=\notdomstate}|=|S^{H^\prime}_{v=\instate}|$, then $S^{H^\prime}_{v=\instate}\cup S^{G-H}_{v=\instate}$ is a minimum $f$-\rsds in $G$.
		\item If $|S^{H^\prime}_{v=\domstate}|<|S^{H^\prime}_{v=\notdomstate}|=|S^{H^\prime}_{v=\instate}|$, then $S^{H^\prime}_{v=\domstate}\cup S^{G-H}_{v=f(v)}$ is a minimum $f$-\rsds in $G$.
		\item If $|S^{H^\prime}_{v=\domstate}|=|S^{H^\prime}_{v=\notdomstate}|<|S^{H^\prime}_{v=\instate}|$, then $S^{H^\prime}_{v=\notdomstate}\cup S^{G-H}_{v=\operatorname{best}(\{f(v), \domstate\})}$ is a minimum $f$-\rsds in $G$.
	\end{itemize}
\end{lemma}
\begin{proof}
	It is easy to see that all claimed sets are $f$-\rsdss in $G$, we now focus on their minimality.
	To this end let $S$ be a minimum $f$-\rsds in $G$.
	
	We begin with the case that $|S^{H^\prime}_{v=\domstate}|=|S^{H^\prime}_{v=\notdomstate}| =|S^{H^\prime}_{v=\instate}|$.
	If $v\in S$, then the first statement holds as $S$ is minimum restricted to $H^\prime$ and therefore it is for free in $S$ when considering~$G-H$.
	So assume~$v\notin S$ and regard $S\cap V(H^\prime)$.
	This set \simdoms all vertices in $H^\prime$ with respect to $f$ except possibly~$v$.
	As $S^{H^\prime}_{v=\domstate}$ is minimum among these sets we have $|S\cap V(H^\prime)|\geq |S_{v=\domstate}^{H^\prime}|= |S_{v=\instate}^{H^\prime}|$ and we can replace $S\cap V(H^\prime)$ by $S_{v=\instate}^{H^\prime}$ without making it larger.
	We now have a minimum $f$-\rsds in $H'$ containing~$v$ and get that $S^{H^\prime}_{v=\instate}\cup S^{G-H}_{v=\instate}$ is also \simdoming with respect to $f$.
	
	Next assume that $|S^{H^\prime}_{v=\domstate}|<|S^{H^\prime}_{v=\notdomstate}|=|S^{H^\prime}_{v=\instate}|$.
	If $v$ is not \simdomed by $S\cap V(H^\prime)$ we are done, so assume $v$ is \simdomed by $S\cap V(H^\prime)$ and hence $|S\cap V(H^\prime)| > |S^{H^\prime}_{v=\domstate}|$.
	If $v\in S$ by Lemma~\ref{lem:maxone1} we have $|S\cap V(G-H)| \geq|S^{G-H}_{v=\instate}|\geq |S^{G-H}_{v=f(v)}|$ and hence,
	\begin{align*}
		|S|= |S\cap V(H^\prime)| + |S\cap V(G-H)| - 1 > |S^{H^\prime}_{v=\domstate}| + |S^{G-H}_{v=f(v)}| - 1.
	\end{align*}
	If $v\notin S$ Lemma~\ref{lem:maxone} implies $|S\cap V(G-H)|\geq |S_{v=\domstate}^{G-H}| \geq |S^{G-H}_{v=f(v)}| - 1$ and we get
	\begin{align*}
		|S|= |S\cap V(H^\prime)| + |S\cap V(G-H)|> |S^{H^\prime}_{v=\domstate}| + |S^{G-H}_{v=f(v)}| - 1.
	\end{align*}
	Both cases then imply $|S|\geq |S^{H^\prime}_{v=\domstate}| + |S^{G-H}_{v=f(v)}|$.
	
	Finally, assume $|S^{H^\prime}_{v=\domstate}|=|S^{H^\prime}_{v=\notdomstate}|<|S^{H^\prime}_{v=\instate}|$.
	First note that this implies~$v\notin S^{H^\prime}_{v=\notdomstate}$.
	Let the set~${S^\prime=(S\setminus V(H)) \cup S^{H^\prime}_{v=\notdomstate}}$.
	Then $|S^\prime|\leq |S|$ and $S^\prime$ is still \simdoming with respect to~$f$.
	Furthermore, it holds that~${|S^\prime\setminus V(H)|\geq S^{G-H}_{v=\operatorname{best}(\{f(v), \domstate\})}}$ and we get 
	\begin{align*}
		|S|\geq |S^\prime|=|S^\prime\setminus V(H)|+|S^\prime\cap V(H)|\geq 
		|S^{H^\prime}_{v=\notdomstate}|+|S^{G-H}_{v=\operatorname{best}(\{f(v), \domstate\})}|,
	\end{align*}
	which implies the desired result.
\end{proof}
Now we show that the algorithm works in fact as desired and argue about the running time.
\begin{theorem}\label{thm::main_algo}
	For a connected graph $G$ and a coloring $f\colon V(G)\to\{\instate,\domstate,\notdomstate\}$, Algorithm~\ref{alg::SDS_STF} correctly computes a minimum size $f$-\rsds $S$ in $G$.
	It can be implemented to run in polynomial time if {\normalfont\textsc{crSDS}} can be implemented to run in polynomial time.
\end{theorem}

Before we start with the proof note that this running time is also called \emph{oracle polynomial} given the oracle \textsc{crSDS}.
As oracle algorithms are not our focus and we just use it here we do not formally introduce this form of algorithms.

\begin{proof}
	
	The proof of correctness can be regarded as a direct consequence of Lemma~\ref{puzzletogether}.
	Nevertheless, we give a formal proof here for the sake of completeness.
	To this end, note that Algorithm~\ref{alg::SDS_STF} can be regarded as a recursive algorithm, where in each step one endpoint except its connection vertex is cut off the graph.
	We do induction on the number of blocks of $G$.
	If $G$ is connected and has no cutvertex the claim trivially holds.
	So let $H'$ be an endpoint of $G$ with connection vertex~$v$ and set~${H= H'-v}$. 
	In the algorithm we now compute~$S^{H'}_{v=i}$ for~$i\in\{\instate,\domstate,\notdomstate\}$.
	By Lemma~\ref{lem:maxone} the three case distinction made in the algorithm (concerning the sizes of these sets) are the only cases that may occur.
	The algorithm now handles the cases as follows:
	\begin{itemize}
		\item If $|S^{H^\prime}_{v=\domstate}|=|S^{H^\prime}_{v=\notdomstate}|=|S^{H^\prime}_{v=\instate}|$, then it adds $S^{H^\prime}_{v=\instate}$ to the current set and maps color $\instate$ to~$v$.
		Thus, by induction the algorithm returns $S^{H^\prime}_{v=\instate}\cup S^{G-H}_{v=\instate}$, which is a minimum $f$-\rsds in $G$ by Lemma~\ref{puzzletogether}.
		\item If $|S^{H^\prime}_{v=\domstate}|<|S^H_{v=\notdomstate}|=|S^{H^\prime}_{v=\instate}|$, then it adds $S^{H^\prime}_{v=\domstate}$ to the current set and leaves the color as it was.
		Thus, by induction the algorithm returns $S^{H^\prime}_{v=\domstate}\cup S^{G-H}_{v=f(v)}$, which is a minimum $f$-\rsds in $G$ by Lemma~\ref{puzzletogether}.
		\item If $|S^{H^\prime}_{v=\domstate}|=|S^{H^\prime}_{v=\notdomstate}|<|S^{H^\prime}_{v=\instate}|$, then it adds $S^{H^\prime}_{v=\notdomstate}$ to the current set and sets the color of $v$ to $\operatorname{best}(\{f(v), \domstate\})$.
		Thus, by induction the algorithm returns $S^{H^\prime}_{v=\notdomstate}\cup S^{G-H}_{v=\operatorname{best}(\{f(v), \domstate\})}$, which is a minimum $f$-\rsds in $G$ by Lemma~\ref{puzzletogether}.
	\end{itemize} 
	As we can see in all considered cases the algorithm correctly computes a minimum $f$-\rsds in $G$.
	
	Considering the running time of Algorithm~\ref{alg::SDS_STF}, note that we can find all blocks in linear time, \textit{cf.}~\cite{hopcroft:73}.
	With a small adjustment of the usual lowpoint algorithm by Hopcroft and Tarjan~\cite{hopcroft:73} we can get the blocks in an order such that each time we regard the next component it is an endpoint of the remaining graph.
	Doing this as a preprocessing step, each call to \textsc{GetEndPoint} takes constant time and the deletion of $H$ is done implicitly.
	In each iteration, besides the three calls to \textsc{crSDS} we only do steps that can be realized in polynomial time, thus, if \textsc{crSDS} can be implemented to run in polynomial time so can Algorithm~\ref{alg::SDS_STF}. 
\end{proof}

\section{\Sdsets on Special Graph Classes}\label{sec::polynomial}
In this section we focus on \sdsets on special graph classes.
In particular, we present some classes, where we can solve \sdsprob in polynomial time.
From Theorem \ref{thm::main_algo} and Theorem \ref{thm::second_algo} we get the following theorem:
\begin{theorem}\label{sds_all_special:thm:poly}
	Let $\mathcal{H}'$ be a class of graphs on which vertex cover is solvable in polynomial time.
	Further, let $\mathcal{H}$ be a class of graphs such that for all $G\in \mathcal{H}$ and subsets $U,W\subseteq V(G)$ with $U\cap W=\emptyset$ we have~${(G-U)-E(G[W])\in\mathcal{H}'}$.
	Then we can solve \sdsprob in polynomial time on graphs from~$\mathcal{H}$.\qed
\end{theorem}
\noindent We now regard some graph classes, where Theorem \ref{sds_all_special:thm:poly} is applicable:
\paragraph{Bipartite Graphs}
Recall that a graph $G$ is bipartite if its vertex set can be partitioned into two sets, such that no edge of $G$ is between vertices in the same set of the partition.
It is easy to see that bipartite graphs are hereditary, \ie, every induced subgraph is again bipartite.
Even if we delete edges in the graph it remains bipartite.
With the help of K\"onig's Theorem \cite{schrijver-book} and a maximum flow algorithm (for example the Hopcroft-Karp algorithm \cite{hopcroft:73}) we can compute a minimum size vertex cover for bipartite graphs in polynomial time.
By Theorem \ref{sds_all_special:thm:poly} Algorithm \ref{alg::SDS_STF} solves \sdsprob on bipartite graphs in polynomial time.
\paragraph{Graphs of Bounded Treewidth}
For a fixed $\kappa\in \mathbb{N}$ regard the class $\mathcal{H}$ of graphs of treewidth at most $\kappa$.
We can find a tree decomposition of graphs in $\mathcal{H}$ in linear time, \cf~\cite{DBLP:journals/tcs/Bodlaender98}.
Arnborg and Proskurowski showed in \cite{ARNBORG198911} that a vertex cover of minimum size can be computed for a graph of bounded treewidth and given tree decomposition in linear time.
As deleting vertices or edges does not increase the treewidth, by Theorem~\ref{sds_all_special:thm:poly} we can compute an \SDset of minimum size in polynomial time for graphs from~$\mathcal{H}$.

Wit bipartite graphs and graphs of bounded treewidth we saw two classes of graphs where $\mathcal{H}=\mathcal{H}'$ in Theorem~\ref{sds_all_special:thm:poly}.
Next we consider the class of graphs where this is not the case and therefore the proof that we can apply Theorem~\ref{sds_all_special:thm:poly} is a bit more involved.

\paragraph{Chordal Graphs}
Recall that a graph $G$ is chordal if any cycle of $G$ with length at least~$4$ contains a chord, \ie, an edge between non subsequent edges in $C$.
Also note that chordal graphs are perfect, \cf~\cite{Diestel17}.
Chordal graphs are hereditary but if we delete edges in a chordal graph, it is possible that the resulting graph is not chordal anymore.
However, with the help of the Strong Perfect Graph Theorem, \cf~Theorem~\ref{strongperfectgraphtheorem} or~\cite{chudnovsky:06}, we can show that the graph after the edge deletion done in Algorithm~\ref{alg::SDS_gen} is perfect.
In perfect graphs we can compute a minimum size vertex cover in polynomial time, \cf~\cite{schrijver-book}.
This leads to a polynomial-time algorithm for solving \sdsprob on chordal graphs by Theorem \ref{sds_all_special:thm:poly}.

It remains to show that for chordal graphs the graph obtained after the edge deletion is perfect.
To do so we use the Strong Perfect Graph Theorem, \cf~Theorem~\ref{strongperfectgraphtheorem} or~\cite{chudnovsky:06} and show that the obtained graph does not contain an odd hole nor an odd anti hole.

\begin{lemma}\label{lemma::chordal->perfect}
	Let $G$ be a chordal graph and $I\subseteq V(G)$. Let $G^{\prime}$ be the 
	graph obtained by deleting all edges between the vertices of $I$ in $G$, 
	\textit{i.e.} \begin{align*}G^{\prime}=G-E(G[I]).\end{align*} 
	Then~$G^{\prime}$ is perfect.
\end{lemma}
\begin{proof}
	Assume $G^{\prime}$ has an odd hole $C_{2k+1}$.
	At most $k$ vertices of $C_{2k+1}$ can be in $I$ since $I$ is an independent set in $G^{\prime}$.
	Hence, there are two consecutive vertices on $C_{2k+1}$ which are not in $I$.
	Since these two vertices do not have a common neighbor in $C_{2k+1}$ and only edges between vertices of $I$ are deleted there exits a cycle in $G$ of length at least four that is contained in $G$ but has no chord.
	In this case $G$ is not chordal which contradicts the assumptions and hence~$G^{\prime}$ cannot have an odd hole.
	
	Now assume that $G^{\prime}$ has an odd antihole $\bar{C}_{2k+1}$, where $C_{2k+1}=u_1\ldots u_{2k+1}u_1$.
	We claim that the subgraph of $G$ induced by $\{u_1,\ldots,u_{2k+1}\}$ has exactly one additional edge in comparison to the subgraph $\bar{C}_{2k+1}$ of $G'$.
	Otherwise if there is no additional edge in $G[\{u_1,\ldots,u_{2k+1}\}]$, then it follows that~$G[\{u_1,\ldots,u_{2k+1}\}]=\bar{C}_{2k+1}$ since no edge is deleted but this contradicts the assumption that $G$ is chordal and hence perfect.
	If there are two or more additional edges, then at least three of the vertices are in~$I$.
	Since all the edges between the vertices in $I$ are deleted, $\bar{C}_{2k+1}$ cannot be an odd antihole in $G^{\prime}$.
	So assume that the additional edge is between $u_2$ and $u_3$ in $G[\{u_1,\ldots,u_{2k+1}\}]$ and thus, these two vertices are the only vertices of $V(\bar{C}_{2k+1})$ in $I$.
	Then the cycle $C=(u_2,u_3,u_1,u_4,u_2)$ is contained in $G$ and has length four but no chord.
	Again this contradicts the assumption that $G$ is chordal and hence $G^{\prime}$ has no odd antihole.
	
	The claim follows by the Strong Perfect Graph Theorem, \cf~Theorem~\ref{strongperfectgraphtheorem} or~\cite{chudnovsky:06}.
\end{proof}

This lemma shows that for a chordal graph the graph obtained after the edge deletion of Algorithm \ref{alg::SDS_gen} is perfect.
We get the following corollary from Theorem~\ref{sds_all_special:thm:poly}:
\begin{corollary}\label{sds_all:special:thm::polytime_solvable}
	In bipartite graphs, chordal graphs and graphs of bounded treewidth we can compute a minimum size \sdset in polynomial time.\qed
\end{corollary}

At this point we refer to Theorem~\ref{thm:sds_np_perfect_graphs} again which states that \sdsprob is \NP-complete when restricted to perfect graphs.
This shows that even though \VCprob is solvable in polynomial time on perfect graphs, Theorem~\ref{sds_all_special:thm:poly} is not applicable to this class.

\section{A 2-Approximation Algorithm for \sdsprob}\label{sec::sds_approx}

For \minVCprob there is an easy $2$-approximation algorithm using maximal matchings, \cf~\cite{Diestel17}.
Together with the result of Theorem \ref{sds_spts:all_spanning_trees:thm:2timesVC} we directly obtain a $4$-approximation for \minsdsprob.
However, we can do better.
In this section we show that there is a $2$-approximation algorithm for \minsdsprob.
The following idea is deduced from another $2$-approximation algorithm for \VCprob using LP-relaxation of the IP-formulation of the problem, \cf~\cite{schrijver-book}.
However, it is more involved than the approximation for \VCprob and therefore worth to be described in detail.
We begin by formulating an inter program for \minsdsprob and prove its correctness.
Then, we use the solution of its LP-relaxation to obtain an integral solution of at most twice the optimal objective function value of the LP and thus at most twice the optimal objective function value of the IP.

(IP~\ref{SDS_spts:2-approx:IP}) describes \minsdsprob for a connected graph $G$.
Let $\CV$ be the set of cutvertices in $G$ and $\NCV\coloneqq V(G)\setminus \CV$.
For every $v\in \CV$ we denote by $\mathcal{B}_v$ the set of blocks of $G$ that contain the vertex $v$.
In a solution the variable $x_v$ states if a vertex~$v$ is in the \SDset or not.
The variable $y_{v,B}$ is only used if $v$ is a cutvertex and states if $v$ is \simdomed by the block $B$.

\begin{mini!}|s|[]
	{x,y}
	{\sum_{v\in V(G)}x_v}
	{\label{SDS_spts:2-approx:IP}}
	{\text{(IP~\ref{SDS_spts:2-approx:IP})}\quad}
	\addConstraint{x_u+x_v}{\geq 1 \quad}{\forall v\in \NCV \text{ and } u\in N_G(v)\label{SDS_spts:2-approx:IP_1}}
	\addConstraint{x_u}{\geq y_{v,B} }{\forall v\in\CV, \forall B\in \mathcal{B}_v,\forall u\in N_B(v)\label{SDS_spts:2-approx:IP_2}}
	\addConstraint{\sum_{B\in\mathcal{B}_v}y_{v,B}+x_v }{\geq 1 }{\forall v\in\CV\label{SDS_spts:2-approx:IP_3}}
	\addConstraint{x_v,y_{v,B}}{\in \{0,1\} \quad}{\forall v\in V(G)\label{SDS_spts:2-approx:IP_4}}
\end{mini!}

\begin{lemma}
	Let $G$ be a graph and $x\in\{0,1\}^{|V(G)|}$. The set~$S=\{v\in V(G)\colon x_v=1\}$ is an \SDset of minimum size if and only if there is a $y$ such that $(x,y)$ is an optimal solution for (IP~\ref{SDS_spts:2-approx:IP}).
\end{lemma}
\begin{proof}
	The lemma follows if we show that the set~$S=\{v\in V(G)\colon x_v=1\}$ is an \SDset if and only if there is a $y$ such that $(x,y)$ is a feasible solution of \eqref{SDS_spts:2-approx:IP_1}--\eqref{SDS_spts:2-approx:IP_4}.
	
	First let $(x,y)$ be a feasible solution for \eqref{SDS_spts:2-approx:IP_1}--\eqref{SDS_spts:2-approx:IP_4} and let $S=\{v\in V(G)\colon x_v=1\}$.
	Note that by \eqref{SDS_spts:2-approx:IP_4} the entries in $x_v$ and $y_{v,B}$ are only $0$ or $1$.
	By \eqref{SDS_spts:2-approx:IP_1} we have for every non-cutvertex that either itself or all its neighbors are in $S$ and thus condition (i) of Theorem~\ref{sdscond} is fulfilled.
	Condition \eqref{SDS_spts:2-approx:IP_3} makes sure that every cutvertex $v$ is in~$S$ or for at least one block $B$ containing $v$ that $y_{v,B}$ has value $1$.
	Hence, together with \eqref{SDS_spts:2-approx:IP_2} all neighbors of $v$ in $B$ are in $S$.
	This implies (ii) of Theorem~\ref{sdscond} and hence it follows that $S$ is a \SDset.
	
	Now assume that $S$ is a \SDset.
	Set $x_v=1$ if $v\in S$ and $x_v=0$ otherwise.
	For every cutvertex $v$ we have that $v$ itself is in $S$ or it is \simdomed, \ie, there is a block $B$ containing $v$ such that all neighbors of $v$ in $B$ are in~$S$.
	We set~$y_{v,B}=1$ if and only if the latter case is true.
	This immediately shows that \eqref{SDS_spts:2-approx:IP_2} and \eqref{SDS_spts:2-approx:IP_3} are fulfilled.
	Condition \eqref{SDS_spts:2-approx:IP_1} is also satisfied by condition~\ref{sdscond_i} of Theorem~\ref{sdscond}.
	This shows that $(x,y)$ is a feasible solution of \eqref{SDS_spts:2-approx:IP_1}--\eqref{SDS_spts:2-approx:IP_4}.
\end{proof}

Next consider the LP-relaxation:
\begin{mini!}|s|[]
	{x,y}
	{\sum_{v\in V(G)}x_v}
	{\label{SDS_spts:2-approx:LP}}
	{\text{(LP~\ref{SDS_spts:2-approx:LP})}\quad}
	\addConstraint{x_u+x_v}{\geq 1 \quad}{\forall v\in \NCV \text{ and } u\in N_G(v)\label{SDS_spts:2-approx:LP_1}}
	\addConstraint{x_u}{\geq y_{v,B} }{\forall v\in\CV, \forall B\in \mathcal{B}_v,\forall u\in N_B(v)\label{SDS_spts:2-approx:LP_2}}
	\addConstraint{\sum_{B\in\mathcal{B}_v}y_{v,B}+x_v }{\geq 1 }{\forall v\in\CV\label{SDS_spts:2-approx:LP_3}}
	\addConstraint{x_v,y_{v,B}}{\geq 0 \quad\quad}{\forall v\in V(G)\label{SDS_spts:2-approx:LP_4}}
\end{mini!}
Let $(x,y)$ be an optimal solution for (LP~\ref{SDS_spts:2-approx:LP}).
We construct a new solution~$(x^{\prime},y^{\prime})$ that is integral in the end and at most doubles the objective function value of $(x,y)$.

The idea is to round at least one variable in \eqref{SDS_spts:2-approx:LP_1} up so \eqref{SDS_spts:2-approx:IP_1} is fulfilled.
It remains to ensure that \eqref{SDS_spts:2-approx:IP_2} and \eqref{SDS_spts:2-approx:IP_3} are satisfied for the cutvertices in $G$.
To do so we use the block graph $T$ of $G$.
We regard the cutvertices of $G$ bottom up in the tree~$T$ and if necessary round up the $x$-variable of the cutvertex itself, while decreasing some $x$-values of neighbors of the cutvertex in order to maintain the approximation quality.
During all rounding steps we ensure that the current solution remains feasible for (LP~\ref{SDS_spts:2-approx:LP}) such that after making all variables integral the resulting solution automatically induces an \SDset.
Further, any variable that is at some point set to $1$ is never changed again, implying that only fractional variables are rounded down.
\paragraph*{First Rounding Step:}
For all $v\in V(G)$ set $x^{\prime}_v:=1$ if $x_v\geq \frac{1}{2}$ and otherwise~$x^{\prime}_v:=x_v$.
Moreover, for each cutvertex $v$ and each block $B$ with $v\in V(B)$ we set
\begin{align}
	\label{SDS_spts:2-approx::first_rounding_y_def}
	y^{\prime}_{v,B}:=\min\{ x^{\prime}_u\colon u\in N_B(v) \}.
\end{align} Whenever we make a change to a variable $x^\prime$ in any rounding step we update all respective variables $y^\prime$ and thus assume that \eqref{SDS_spts:2-approx::first_rounding_y_def} remains valid throughout the procedure.

After the first rounding step, all constraints~\eqref{SDS_spts:2-approx:IP_1} are already fulfilled as by~\eqref{SDS_spts:2-approx:LP_1} one of the two variables is greater or equal to $\frac{1}{2}$.
Since we never decrease a variable with value $1$ this does not change during the preceding rounding steps.
Further, note that all variables now have a value of $1$ or less than $\frac{1}{2}$.
We keep this invariant throughout the remaining rounding steps.

Now regard the block graph $T$ of $G$ and root it at any cutvertex.
It is easily observed that we may now recursively choose a cutvertex $v$ such that all descendants of $v$ in~$T$ that are cutvertices have already been regarded.
If for some block $B$ containing $v$ we have $y^\prime_{v,B}\geq\frac{1}{2}$, then by \eqref{SDS_spts:2-approx::first_rounding_y_def} it holds that $y^\prime_{v,B}=1$, which implies that the vertex~$v$ is \simdomed by block $B$ and we can safely go to the next cutvertex.
So assume that~$y^\prime_{v,B}<\frac{1}{2}$ for all blocks containing $v$.
We denote by $B^\prime$ the parent of $v$ in~$T$ and by $B_1,\dots, B_k$ its children.
As $y^\prime_{v,B^\prime}<\frac{1}{2}$, by constraint~\eqref{SDS_spts:2-approx:LP_2} it holds that 
\begin{align*}
	x^\prime_v + \sum_{i=1}^{k}y^\prime_{v,B_i}\geq\frac{1}{2}.
\end{align*}
For every $i=1,\dots,k$ there exists some node $u_i$ fulfilling $x_{u_i}^\prime=y^\prime_{v,B_i}$ by~\eqref{SDS_spts:2-approx::first_rounding_y_def}.
We can use these vertices to define our next rounding step.

\paragraph*{Second Rounding Step} For every cutvertex $v$ moving bottom up in the block graph~$T$ of $G$, test if $y^\prime_{v,B}\geq\frac{1}{2}$ for some block $B$ containing $v$.
If none such block exists, set $x^\prime_{v}=1$ and $x_{u_i}^\prime=0$ for all $i=1,\dots, k$.

Note that after each of these rounding steps if we increase $x_v^\prime$ we may safely set $y^\prime_{v,B_i}$ to~$0$, as the constraint~\eqref{SDS_spts:2-approx:IP_3} is satisfied due to $x^\prime_v=1$.
Thus, all constraints~\eqref{SDS_spts:2-approx:IP_2} corresponding to the cutvertex $v$ are satisfied after the rounding step.
Further, decreasing variables that have value less than~$\frac{1}{2}$ does not violate any constraint, as all constraints corresponding to vertices in the children of $v$ are satisfied solely by variables that are already set to $1$.
With these arguments we can be sure that after any second rounding step, the solution remains feasible.
Note that it is possible that we have to update some $y^\prime$ variables, as we changed the value of some $x^\prime$ variables and the minimum in~\eqref{SDS_spts:2-approx::first_rounding_y_def} may have changed. 

We argue later that these rounding steps do not increase the objective value of the current solution too much.

\paragraph*{Third Rounding Step}
After iterating through all cutvertices we set all remaining fractional variables to $0$.
\begin{theorem}
	The described algorithm is a $2$-approximation algorithm for \minsdsprob and runs in polynomial time.
\end{theorem}
\begin{proof}
	First we show that the objective value of $(x',y')$ is at most twice the optimal objective value of (IP~\ref{SDS_spts:2-approx:IP}).
	In every first or second rounding step we replace the value of a subset of variables, which have summed up value at least~$\frac{1}{2}$, by the value $1$.
	This clearly implies that the defined solution has objective value at most twice the objective value of the optimal LP solution.
	
	We now show that $(x^{\prime},y^{\prime})$ is a feasible solution for (IP~\ref{SDS_spts:2-approx:IP}).
	All entries in $x^{\prime}$ and $y^{\prime}$ are integral.
	In~\eqref{SDS_spts:2-approx:LP_1} $x_u$ or $x_v$ was greater or equal to $\frac{1}{2}$ and hence $x^{\prime}_u$ or~$x^{\prime}_v$ was set to $1$.
	We do not decrease it later on, so \eqref{SDS_spts:2-approx:IP_1} is satisfied.
	Moreover, we made sure that for every cutvertex $v$ at least one of the variables $x^{\prime}_v$ or $y^{\prime}_{v,B}$ for some block $B$ containing~$v$ equals $1$ and hence, \eqref{SDS_spts:2-approx:IP_3} is fulfilled.
	Condition \eqref{SDS_spts:2-approx:IP_2} is also satisfied since we set $y^{\prime}_{v,B}$ only to $1$ if all the corresponding $x_u$ equal $1$ otherwise we set it to $0$.
	
	This shows that $(x^{\prime},y^{\prime})$ is a feasible solution for (IP~\ref{SDS_spts:2-approx:IP}) that has at most twice the value of the objective function value of an optimal solution of (LP~\ref{SDS_spts:2-approx:LP}) and hence, of (IP~\ref{SDS_spts:2-approx:IP}).
	
	We need polynomial time to set up and solve (LP~\ref{SDS_spts:2-approx:LP}), \cf~\cite{schrijver-book}.
	All rounding steps can be implemented to run in polynomial time.
\end{proof}

\section{Conclusion}

We considered the problem of simultaneously dominating every spanning tree in a graph.
We proved that in a $2$-connected graph a subset of the vertices is a \sdset if and only if it is a vertex cover.
Although finding a minimum vertex cover and finding a minimum \sdset is thereby strongly related, crucial differences remain. 
On the one hand we proved that the size of a minimum \sdset and the size of a minimum vertex cover may differ by a factor of $2$ and that this bound is tight. 
On the other hand, we proved that \sdsprob is \NP-complete on perfect graphs, whereas \VCprob is polynomial time solvable. 
This also implies that \sdsprob is solvable in polynomial time on $2$-connected, perfect graphs. 
Afterwards, we presented an algorithm that solves \sdsprob by decomposing it into smaller subproblems that can be solved by some preprocessing and an oracle for the \minVCprob.
We argued that the algorithm can be implemented to run in linear time, when the input graph is restricted to bipartite graphs, chordal graphs, or graphs of bounded treewidth.
Finally, we presented a $2$-approximation based on LP-rounding.

\newpage
\bibliographystyle{alpha}
\bibliography{Verzeichnis}
\label{sec:biblio}

\end{document}